\documentclass[10pt]{amsart}
\usepackage{graphicx}

\newtheorem{theorem}{Theorem}[section]
\newtheorem{proposition}[theorem]{Proposition}
\newtheorem{lemma}[theorem]{Lemma}

\theoremstyle{remark}

\newtheorem{remark}{Remark}

\begin{document}

\title[Carleman polynomials]{Asymptotics of Carleman polynomials for level curves of the inverse of a shifted Zhukovsky transformation}

\author[P. Dragnev]{Peter Dragnev}
\email{dragnevp@ipfw.edu}
\address{Indiana-Purdue University Fort Wayne, Department of Mathematical Sciences,
2101 E. Coliseum Boulevard, Fort Wayne, IN 46805-1499, USA}

\author[E. Mi\~{n}a-D\'{\i}az]{Erwin Mi\~{n}a-D\'{\i}az}
\email{minadiaz@olemiss.edu}
\address{University of Mississippi, Department of Mathematics, Hume Hall 305, P. O. Box 1848,  University, MS  38677-1848, USA}

\author[M. Northington V]{Michael Northington V}
\email{michael.c.northington.v@vanderbilt.edu}
\address{Vanderbilt University, Department of Mathematics,
1326 Stevenson Center, Nashville, TN 37240, USA}
\thanks{The third author conducted his research while at the University of Mississippi as a GAANN fellow.}
%\dedicatory{}
\begin{abstract}
This paper complements the recent investigation of \cite{DM} on the asymptotic behavior of polynomials orthogonal over the interior of an analytic Jordan curve $L$. We study the specific case of $L=\{z= w-1 +(w-1)^{-1},\ |w|=R\}$, for some $R>2$, providing an example that exhibits the new features discovered in \cite{DM}, and for which the asymptotic behavior of the orthogonal polynomials is established over the entire domain of orthogonality. Surprisingly, this  variation of the classical example of the ellipse turns out to be quite sophisticated. After properly normalizing the corresponding orthonormal polynomials $p_n$, $n=0,1,\ldots$, and on certain critical subregion of the orthogonality domain, a subsequence $\{p_{n_k}\}$ converges if and only if $\log_{\mu^4}(n_k)$ converges modulo 1 ($\mu$ being an important quantity associated to $L$). As a consequence, the limiting points of the sequence $\{p_n\}$ form a one parameter family of functions, the parameter's range being the interval $[0,1)$. The 
polynomials $p_n$ are much influenced by a certain integrand function, the explained behavior being the result of this integrand having a nonisolated singularity that is a cluster point of poles.  The nature of this singularity sparks purely from geometric considerations,
 as opposed to the more common situation where the critical singularities come from the orthogonality weight.
\end{abstract}
\maketitle
\section{Introduction and new results}
The study of orthogonal polynomials over planar regions seems to have originated in the work of T. Carleman \cite{Carleman}, followed up by contributions from several authors, but more prominently by P. K. Suetin (see his monograph \cite{Suetin} and the many references therein).

Recently, the subject has experienced a new surge, with many new interesting results in a variety of topics such as the asymptotic behavior and zero distribution of the orthogonal polynomials \cite{DM, Gust, Levin, Maymeskul, Mina4, Mina3, Mina2, SY, Simanek2}, universality and Christoffel functions \cite{Lubinsky,SY,Totik}, and the existence of recurrence relations \cite{BSS, Sty1, Sty2}. In particular, \cite{Gust} and \cite{Totik} consider orthogonality over several domains, while \cite{SY} considers orthogonality with respect to certain potential theoretic varying weights. The papers \cite{Saff} and \cite{Simanek1}, though more general in scope, also discuss important implications for planar orthogonality.

Let $G_1$ be a bounded simply-connected domain of $\mathbb{C}$, whose boundary $L_1$ is an analytic Jordan curve, and let $\{p_n(z)\}_{n=0}^{\infty}$ be the unique sequence of polynomials satisfying that $p_n$ is a polynomial of degree $n$ with positive leading coefficient, and that
\begin{equation*}
\frac{1}{\pi}\int_{G_1}p_n(z)\overline{p_m(z)}dxdy=\begin{cases}
0, & n\not=m,\\
1, & n=m.
\end{cases}
\end{equation*}

These are the polynomials originally investigated by Carleman in \cite{Carleman}. Among other things, Carleman derived an asymptotic formula that establishes the behavior of $p_n(z)$ as $n\to\infty$ on certain neighborhood of $\overline{\mathbb{C}}\setminus G_1$ (the set denoted by $\Omega_\rho$ below). To state this result with precision, we first need to convene on some notation.

For each $r>0$, we define
\begin{equation*}
	\mathbb{T}_r:=\{w:|w|=r\}, \quad  \Delta_r := \{w: r<|w| \leq \infty\}, \quad  \mathbb{D}_r:=\{w:|w|<r\}.
\end{equation*}

Let $\Omega_1$ be the unbounded component of $\overline{\mathbb{C}}\setminus L_1$, and let $\psi(w)$ be the unique conformal map of $\Delta_1$ onto $\Omega_1$ that satisfies $\psi(\infty)=\infty$ and $\psi'(\infty)>0$.  Because $L_1$ is analytic, there is a smallest number $0\leq \rho <1$ for which $\psi$ admits an analytic and univalent continuation to $\Delta_\rho$, and we define
\[
\phi(z): \Omega_\rho \rightarrow \Delta_\rho
\]
to be the inverse of $\psi$.

Finally, for each $\rho \leq r< \infty$, define
\begin{align*}
	\Omega_r:= \psi(\Delta_r), \quad  L_r := \partial \Omega_r, \quad  G_r:= \mathbb{C} \setminus \overline{\Omega}_r,
\end{align*}
so that for $r>\rho$, $L_r$ is an analytic Jordan curve.

Carleman proved that
\begin{equation}\label{carlemanformula}
\lim_{n\to\infty}\frac{p_n(z)}{\sqrt{n+1}[\phi(z)]^n}=  \phi'(z), \quad z\in\Omega_\rho,
\end{equation}
the convergence being uniform on compact subsets of $\Omega_\rho$  (for a more complete statement, see \cite[Satz IV]{Carleman}, \cite[Sec. 1]{Gaier}, and also \cite{series}).

This establishes the asymptotic behavior of $p_n(z)$ on the closed exterior $\overline{\Omega}_1$ of $L_1$, and on a portion of its interior $G_1$, namely, on the ``strip" $\Omega_\rho\cap G_1$. What happens at the remaining points of $G_1$ has been recently investigated in \cite{DM,Mina4}. In turns out that there is a subset $\Sigma_1\subset G_1$, which is, in general, larger than the strip $\Omega_\rho\cap G_1$, on which an asymptotic formula just like (\ref{carlemanformula}) holds true. This set $\Sigma_1$ is, however, less straightforward to define, and its construction depends on a conformal map $\varphi(z)$ of $G_1$ onto $\mathbb{D}_1$.

Such a conformal map $\varphi$ has a meromorphic and univalent continuation to $G_{1/\rho}$ (see \cite{DM} for details), so that the composition $\varphi(\psi(w))$ is a well-defined meromorphic function in the annulus $\rho<|w|<1/\rho$. We can then define the important quantity $\mu\geq 0$ to be the smallest number such that $\varphi(\psi(w))$ has a meromorphic continuation, denoted by $h_{\varphi}(w)$, to the annulus $\mu < |w| < 1/\rho$.

We let $\Sigma$ be the set of points $z \in G_1$ such that the equation
\begin{equation}\label{eq1}
h_{\varphi}(w)=\varphi(z)
\end{equation}
has at least one solution in the annulus $\mu< |w|<1$, and let $\Sigma_0:=G_1\setminus \Sigma$.

For fixed $z \in \Sigma$, of the solutions that the equation $(\ref{eq1})$ has in $\mu<|w|<1$, only finitely many (say $s$ of them)  have largest modulus, and we denote these solutions of largest modulus by $\omega_{z,1}, \ldots, \omega_{z,s}$. Letting $\alpha_{z,k}$ denote the multiplicity of  $h_{\varphi}$  at $\omega_{z,k}$, we associate to each integer $p\geq 1$ the set
\begin{equation}\label{eq115}
\Sigma_p:=\{z\in \Sigma :  \alpha_{z,1}+ \ldots +\alpha_{z,s}=p\}\,.
\end{equation}

Thus, $\Sigma_1$ consists of those points $z \in \Sigma$ such that the equation $(\ref{eq1})$ has one solution of largest modulus, and this solution is simple. Finally, we define the functions $\Phi:\Sigma_1\rightarrow \{w: \mu<|w|<1\}$ and $r:G_1\rightarrow [\mu,1)$ by
\begin{equation*}
\Phi(z):=\omega_{z,1}, \quad z \in \Sigma_1,
\end{equation*}
and
 \begin{equation*}
	r(z):=\begin{cases}
			|\omega_{z,1}|, &z \in \Sigma,\\
			\mu, &z \in \Sigma_0.		
\end{cases}
\end{equation*}

It is not difficult to see (see Lemma 11 and Corollary 12 of \cite{DM}) that neither $\mu$ nor the $\Sigma_p$ sets depend on the interior conformal map $\varphi$ chosen. Moreover, $\Sigma$ and $\Sigma_1$ are open, $\Phi(z)$ is analytic and univalent, and $r(z)$ is continuous.

Notice that for $z \in \Omega_{\rho} \cap G_1$,
\[
h_{\varphi}(\phi(z))=\varphi(\psi(\phi(z)))= \varphi(z),
 \]
so that
\begin{equation*}
\Phi(z)=\phi(z),\quad z \in \Omega_{\rho} \cap G_1,
\end{equation*}
and therefore, $\Sigma_1 \supset \Omega_{\rho}\cap G_1$. In general, $\Sigma_1$ is larger that $\Omega_{\rho}\cap G_1$.

Using this partition of $G_1$ into $\Sigma_p$ sets, it was proven in \cite{DM} that
\begin{equation}\label{eq5}
\lim_{n\to \infty}\frac{p_n(z)}{\sqrt{n+1} [\Phi(z)]^n}=\Phi'(z),\quad z\in\Sigma_1,
\end{equation}
the convergence being uniform on compact subsets of $\Sigma_1$, and that
\begin{equation}\label{eq101}
\limsup_{n\rightarrow \infty} |p_n(z)|^{1/n} = r(z), \quad z\in G_1\setminus \Sigma_1.
\end{equation}

Both (\ref{eq5}) and (\ref{eq101}) can be obtained from the following integral representation, which is fundamental as well for proving the results of this paper.
\begin{proposition} \label{cor1} Let $r$ be any fixed number satisfying that $\rho<r<1$. Then, for every integer $n$ sufficiently large,
\begin{equation}\label{eq4}
p_n(z)=  \frac{\sqrt{n+1}\varphi'(z)}{2\pi i}\oint_{\mathbb{T}_1} \frac{w^{n}(1+K_n(w))  dw}{h_\varphi(w)-\varphi(z)},\quad z \in G_1,
\end{equation}
where $K_n(w)$ is analytic in $|w|<1/r$  and $K_n(w)=O(r^{2n})$ locally uniformly as $n\to\infty$ on $|w|<1/r$.
\end{proposition}

A first version of  Proposition \ref{cor1} was proven in \cite{Mina4}. The version stated above is simpler to use and we shall briefly indicate at the end of Section \ref{Proofs} below how to derive it from the recent results of \cite{series}.

Roughly speaking, (\ref{eq4}) is telling us that for each fixed $z\in G_1$, $p_n(z)$ behaves as $n\to\infty$ like the $-(n+1)$th coefficient of the Laurent expansion that the function $w\mapsto [h_\varphi(w)-\varphi(z)]^{-1}$ has in the annulus $\rho<|w|<1$. If $z\in \Sigma_1$, the Laurent expansion encounters on its inner circle of convergence just one singularity, which happens to be a simple pole, thereby  implying (\ref{eq5})\footnote{If $z\in \Sigma_p$, $p\geq 2$,  the first singularities are also finitely many poles, but they have a total multiplicity larger than $1$, and although this certainly leads to a better estimate than  (\ref{eq101}), that estimate is essentially pointwise, unless more is known about the particularities of the orthogonality domain.}.

If $z\in \Sigma_0$,  however, the inner circle of convergence of the Laurent expansion is $\mathbb{T}_\mu$, where the function $h_\varphi(w)$ encounters its first nonpolar singularity. Given that the behavior of $h_\varphi(w)$ on $\mathbb{T}_\mu$ can be wildly erratic, (\ref{eq101}) is, in general, the best we can say for points $z\in \Sigma_0$.

Nonetheless, for more specific orthogonality domains one should be able to say more than just (\ref{eq101}), and it is the purpose of this paper to provide a ``full featured" orthogonality domain for which the corresponding set $\Sigma_1$ is \emph{larger} than $\Omega_\rho\cap G_1$,  \emph{the interior of $\Sigma_0$ is nonempty}, and we can establish the strong asymptotic behavior of $p_n(z)$ for every point $z\in \Sigma_0$.

It turns out, however, that providing such an example is much trickier than it might seem at first sight. The difficulty lies in that, when trying to guarantee that $\Sigma_1$ be larger than $\Omega_\rho\cap G_1$, we lose control of the nature of the first \emph{nonpolar} singularities that the function $h_\varphi(w)$ encounters.

 %But it is in fact the driving force of this paper to see what type of behavior can this rare singularities %impose on the orthogonal polynomials.

We take for orthogonality domain $G_1$ the interior of the image by the Zhukovsky transformation of a circle $C_R$ centered at $-1$ of radius $R>2$ (see Figure \ref{figura1}). In other words, the boundary
\begin{equation}\label{eq118}
L_1:=\left\{w-1+(w-1)^{-1}:|w|=R\right\}
\end{equation}
of $G_1$ is a level curve of the inverse of the shifted Zhukovsky transformation $w\mapsto w-1+(w-1)^{-1}$.
\begin{figure}
\centering
\includegraphics[scale=.7]{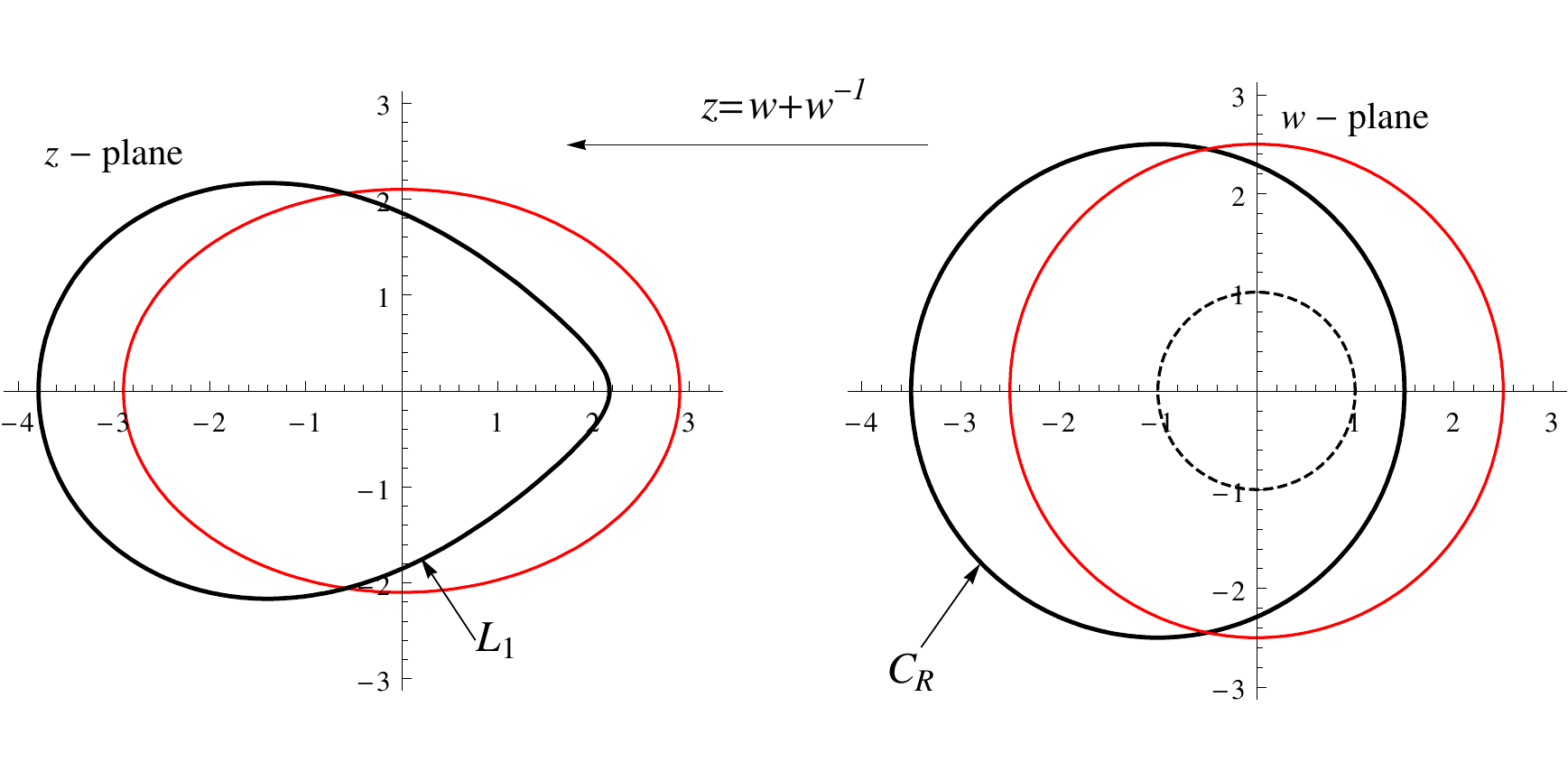}
\caption{$L_1$ is the image by the Zhukovsky transformation of the circle $C_R=\{w:|w+1|=2.5\}$.}
\label{figura1}
\end{figure}
\begin{figure}
\centering
\includegraphics[scale=.7]{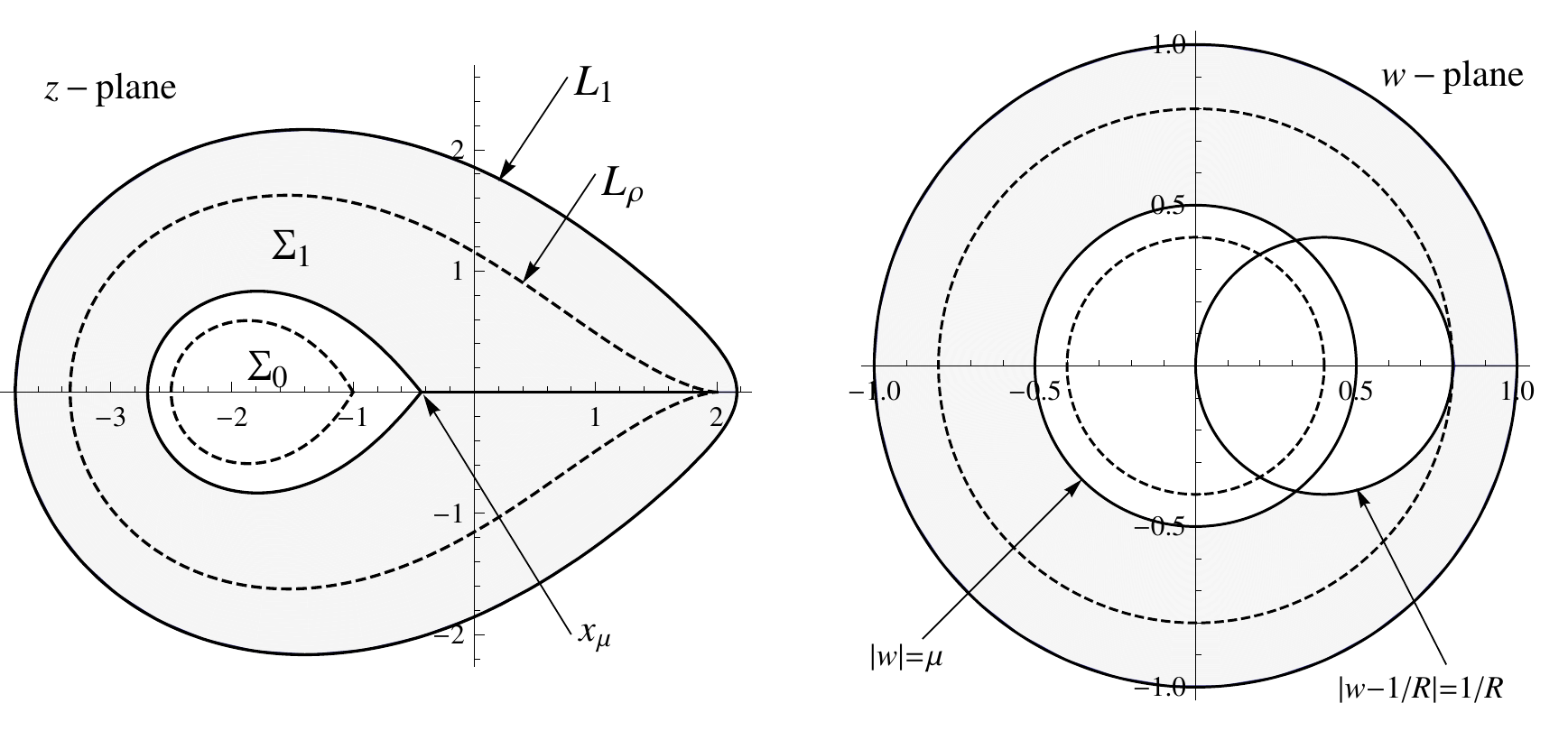}
\caption{Sets $\Sigma_1$, $\Sigma_2$ and $\Sigma_0$ for the curve $L_1$ defined in (\ref{eq118}) for $R=2.5$. $\Sigma_1$ is the greyish region, $\Sigma_2$ is the segment $(x_\mu,2]$, and $\Sigma_0$ is the white region together with its boundary.}\label{figura3}
\end{figure}

To avoid unnecessary complications, we shall  refer to Figure \ref{figura3} and content ourselves with a visual understanding of the geometric aspects of the curve $L_1$ defined by (\ref{eq118}), making it all precise in Proposition \ref{thm1} of the next  section.

For this curve we have that $G_1=\Sigma_0\cup\Sigma_1\cup\Sigma_2$,  where $\Sigma_1$ is the full greyish region, $\Sigma_2$ is the half-open segment $(x_\mu,2]$, and $\Sigma_0$ is the white region together with its boundary. The set $\Omega_\rho\cap G_1$ is the strip between $L_1$ and the dotted line $L_\rho$. Also, for this curve we have
\[
 \mu=\frac{R-\sqrt{R^2-4}}{2}.
\]

\begin{figure}
\centering
\includegraphics[scale=.7]{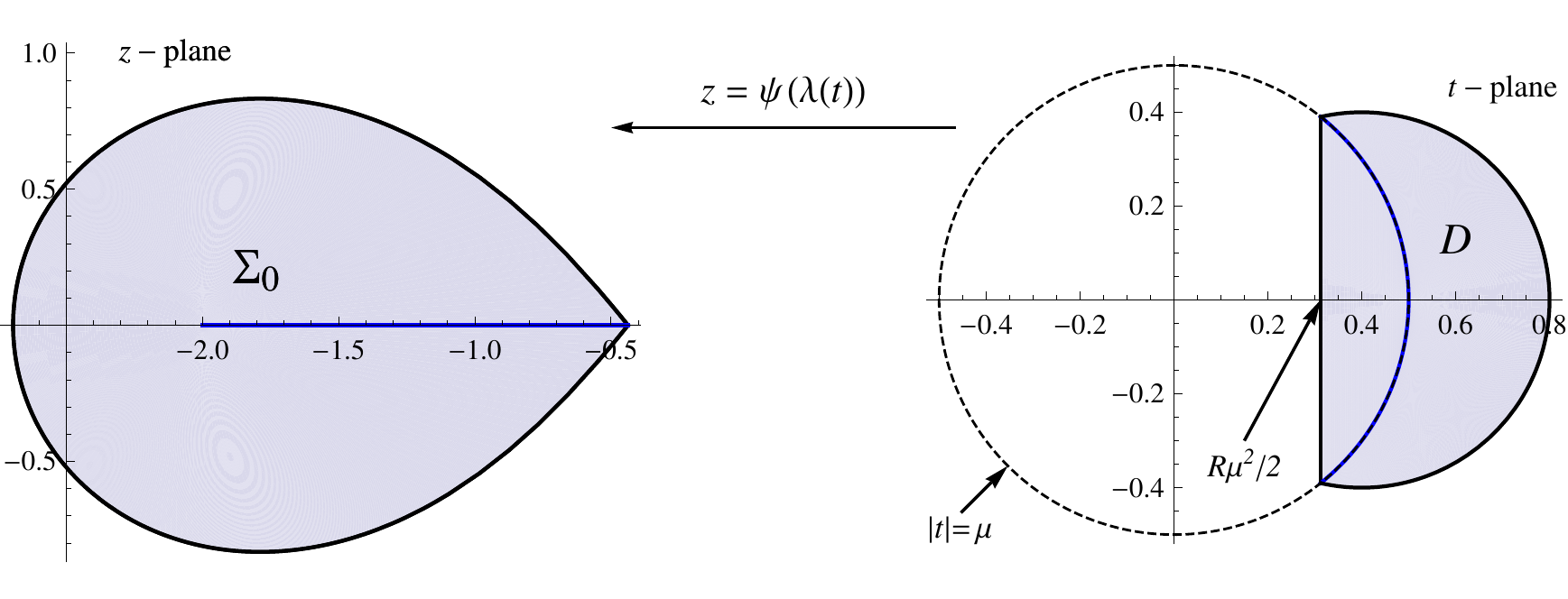}
\caption{The region $D$ is symmetric about the circle $|t|=\mu$, and is mapped by $\psi(\lambda(t))$ onto the interior of $\Sigma_0$ in a two-to-one fashion.}\label{RegionD}
\end{figure}

What makes this example intriguing is that the corresponding function  $h_\varphi(w)$ encounters on the circle of radius $\mathbb{T}_\mu$ a non-isolated singularity of ``essential type", in the sense that in every punctured neighborhood of it, $h_\varphi(w)$ attains every value of the extended complex plane. This singularity  imposes on the $p_n$'s a very interesting behavior, described by Theorems \ref{thm8} and \ref{mainthm} below. We emphasize that the nature of this singularity sparks purely from geometric considerations.

Two essential components to describe the behavior of $p_n$ on the interior of $\Sigma_0$ are the M\"{o}bius transformation $\lambda(z)$ and the doubly infinite series $\chi(t)$ defined by
\begin{equation}\label{fractional}
\lambda(z):=\frac{z-\mu}{\mu z-1}, \qquad \chi(t):=t\sum_{k=-\infty}^{\infty}\mu^{4k}e^{(\mu-\mu^{-1})\mu^{4k}t}, \quad  \Re(t)>0.
\end{equation}

For the curve $L_1$ given by (\ref{eq118}), the exterior conformal map $\psi$ of $\Delta_1$ onto $\Omega_1$ is given by 
\begin{equation}\label{mapping}
\psi(w)=Rw-1+\frac{1}{Rw-1}, \quad w\in\overline{\mathbb{C}},
\end{equation}
and the composition $\psi(\lambda(t))$ maps the region 
\begin{equation}\label{domainD}
D:=\left\{t:\left|t-1/ R\right|<1/ R,\ \Re(t)>R\mu^2/2\right\}
\end{equation}
in a two-to-one fashion onto the interior of $\Sigma_0$ (see Figure \ref{RegionD}). The asymptotics of $p_n(z)$ for $z$ in the interior of $\Sigma_0$ take a simpler and more elegant form if they are stated instead on the region $D$ by means of  the functions 
\[
\mathcal{P}_n(t):=p_n(\psi(\lambda(t))),\quad n\geq 0.
\]

Hereafter the  fractional part of a number $x$ will be denoted by $\langle x\rangle$.

\begin{theorem}\label{thm8} Let $\{n_k\}_{k=1}^\infty$ be a subsequence of the natural numbers. The sequence  $\{\sqrt{n_k}\mu^{-n_k}\mathcal{P}_{n_k}\}_{k=1}^\infty$ converges normally on $D$ if and only if
\begin{equation}\label{limite}
\lim_{k\to\infty}e^{2\pi i(\log_{\mu^4}(n_k)-q)}=1
\end{equation}
for some $q\in[0,1)$, in which case
\begin{equation}\label{ecuacion1}
\lim_{k\to\infty}\frac{\sqrt{n_k}\,\mathcal{P}_{n_k}(t)}{\mu^{n_k}}=f_q(t):=\frac{(1-\mu t)^2(1-\mu^3/t)^2}{(1-\mu^4)R} \cdot\frac{\chi(\mu^{4q}t)-\chi(\mu^{4q+2}/t)}{t-\mu^2/t}.
\end{equation}

Moreover,  $f_q\not=f_p$ for $0\leq q<p<1$, and since the sequence $\{\langle\log_{\mu^4}n\rangle\}_{n\in\mathbb{N}}$ is dense in $[0,1)$, it follows that the family $\{f_q(t):0\leq q<1\}$ comprises all the limit points that the sequence $\{\sqrt{n}\mu^{-n}\mathcal{P}_n\}_{n=0}^\infty$ has on $D$.
\end{theorem}

Thus, $p_n(z)$ decreases like $\mu^n/\sqrt{n}$ on the interior of $\Sigma_0$, but $\lim_{n\to\infty }\sqrt{n}\mu^{-n}p_n(z)$ only exists through those subsequences $\{n_k\}$ for which $\log_{\mu^4}(n_k)$ converges modulo 1. 

Theorem \ref{thm8} follows by combining the following asymptotic formula with the integral representation for $\chi(t)$ given in (\ref{fractional1}).
\begin{theorem}\label{mainthm}
 The asymptotic representation
\begin{align}\label{ecuacion2}
\begin{split}
\mathcal{P}_{n}(t)={}&\frac{\mu^{n}}{\sqrt{n}}\cdot\frac{(1-\mu t)^2(1-\mu^3/t)^2}{(1-\mu^4)R}\left(\frac{\chi(nt)-\chi(n\mu^2/t)}{t-\mu^2/t}+O(1/n)\right)
\end{split}
\end{align}
holds true locally uniformly  on $D$ as $n\to\infty$.
\end{theorem}

Combining this theorem with (\ref{chi}) we obtain that  
\[
\frac{\sqrt{n+1}\,p_{n+1}(z)}{\mu^{n+1}}-\frac{\sqrt{n}p_{n}(z)}{\mu^{n}}=O(1/n)
\]
locally uniformly on the interior of $\Sigma_0$.

\begin{remark} For each $z\in\mathbb{C}$, the solutions $t_{z,\pm}$ to the equation $z=\psi(\lambda(t))$ are given by 
\[
t_{z,\pm}=\mu\cdot\frac{4-(R^2-2)z \mp R\sqrt{R^2-4}\sqrt{z^2-4}}{2(R^2-2-z)},
\]
which satisfy that $t_{z,+}=\mu^2/t_{z,-}$. Thus, replacing $t$ in (\ref{ecuacion1}) and (\ref{ecuacion2}) by anyone of these two solutions yields the corresponding asymptotic statements in the $z$-plane, more specifically, on the interior of $\Sigma_0$.  
\end{remark}

\begin{remark} Let $Z$ be the set of all $z_0\in\overline{\mathbb{C}}$ with the property that every neighborhood of $z_0$ contains zeros of infinitely many $p_n$'s. From the results of \cite{DM}, we infer that $\partial \Sigma_0\cup\Sigma_2\subset Z\subset \Sigma_0\cup\Sigma_2$. The points of $Z$ lying in the interior of $\Sigma_0$ are the images under $\psi(\lambda(t))$ of the zeros that the functions $$\frac{\chi(\mu^{4q}t)-\chi(\mu^{4q+2}/t)}{t-\mu^2/t},\quad q\in [0,1),$$ have in the region $D$. Numerical computations indicate that for at least some values of $q$ and $\mu$, these functions do have zeros in $D$.
\end{remark}

\begin{remark} The type of singularity encountered by the function $h_\varphi$ of our example seems to occur with frequency. This is a consequence of the meromorphic continuation properties that $h_\varphi$ displays in general, as explained in Propositions 4, 5 and 6 of  \cite{DM}. Such is the case, for instance, of the level curve
\[
L_1=\left\{z=w+\frac{1}{2w^2}:|w|=R\right\},\quad R>1
\]
(for $R=1$, this curve is the hypocycloid of three cusps). The example discussed in this paper, though already complex, is the simplest we have found.

\end{remark}

Our next result establishes the behavior of $p_n(z)$ at the remaining points of $G_1$.
\begin{theorem}\label{thm4} The estimate
\begin{equation}\label{parta}
\frac{(2R)^{n+1}p_n(z)}{\sqrt{n+1}(z+2+\sqrt{z^2-4})^{n}}= \frac{z+\sqrt{z^2-4}}{\sqrt{z^2-4}}+O(1/n)
\end{equation}
holds uniformly  as $n\to\infty$ on compact subsets of $\Sigma_1\cup\partial \Sigma_0\setminus\{x_\mu\}$.

For every $z=2\cos\theta$,  $0\leq \theta\leq \arccos(x_\mu/2)$, $x_\mu=(1+\mu^2)^2-2$, we have
\begin{equation}\label{partb}
p_n(z)= \sqrt{n+1}\left(2/R\right)^{n+1}\cos^n(\theta/2)\left\{\frac{\sin( (n+2)\theta/2)}{2\sin\theta}+\epsilon_n(z)\right\},
\end{equation}
where $\epsilon_n(z)$ decays geometrically fast on compact subsets of $(x_\mu,2]$, while $\epsilon_n(z)=O(1/n)$ uniformly on $[x_\mu,2]$ as $n\to\infty$.
\end{theorem}

The rest of the paper is organized as follows. First, in Section 2 we acquire a good understanding of the meromorphic continuation properties of the associated function $h_\varphi$. This will allows us to represent the integral in (\ref{eq4}) as an infinite, $n$-dependent sum of residues. In Section 3 we derive several lemmas needed to estimate the asymptotic behavior as $n\to\infty$ of such infinite sums, and finally in Section \ref{Proofs}, we prove Theorems \ref{thm8}, \ref{mainthm} and \ref{thm4}, briefly indicating at the end how to derive Proposition \ref{cor1} from the recent results of \cite{series}.

\section{Meromorphic continuation of $h_\varphi$}

For a fixed value of $R>2$, let $L_1$ be given by (\ref{eq118}). From very well-known properties of the Zhukovsky transformation $w \mapsto w+1/w$, it follows that $L_1$ is an analytic Jordan curve, and that the function $\psi(w)$ given by (\ref{mapping}) maps $\Delta_1$ conformally onto the exterior $\Omega_1$ of $L_1$.  

Moreover, $\psi$ maps both $\{w:|w-1/R|>1/R\}$ and $\{w:|w-1/R|<1/R\}$ conformally onto $\overline{\mathbb{C}}\setminus [-2,2]$, while mapping both the closed upper  and lower halves of the circle $|w-1/R|=1/R$ univalently onto $[-2,2]$. Hence,
\[
\rho=2/R,
\]
and $L_\rho$ is the image by $\psi$ of the circle $\mathbb{T}_{2/R}$ (see Figure \ref{figura3}).

For every $z\in \mathbb{C}$, the equation $z=\psi(w)$ has for solutions the numbers 
\begin{equation}\label{eqq33}
v_{z,+}:=\frac{z+2+\sqrt{z^2-4}}{2R}, \quad v_{z,-}:=\frac{z+2-\sqrt{z^2-4}}{2R},\quad z\in\mathbb{C},
\end{equation}
where we denote by $\sqrt{z^2-4}$ the branch of the square root of $z^2-4$ in $\mathbb{C}\setminus[-2,2]$ that is positive along $(2,\infty)$, extended to $[-2,2]$ by taking its boundary values from the upper half plane.

When $z\in \mathbb{C}\setminus[-2,2]$, $v_{z,+}$ and $v_{z,-}$ lie, respectively, outside and inside the circle $|w-1/R|=1/R$, and consequently
\[
|v_{z,+}|=\left|R^{-1}+\frac{R^{-2}}{v_{z,-}-R^{-1}}\right|=\left|\frac{R^{-1}}{v_{z,-}-R^{-1}}\right||v_{z,-}|>|v_{z,-}|,
\]
that is, $|v_{z,+}|>|v_{z,-}|$ for every  $z\in\mathbb{C}\setminus [-2,2]$. Of course, if $z\in[-2,2]$, then $v_{z,+}=\overline{v}_{z,-}$, they lie on the circle  $|w-1/R|=1/R$, and $v_{z,+}=v_{z,-}$ if and only if $z=\pm2$.

It follows that the inverse of $\psi(w)$ is the function
\[
\phi(z)=v_{z,+}, \quad z\in \Omega_\rho,
\]
which is indeed analytic and univalent all over $\mathbb{C}\setminus [-2,2]$.

As for the corresponding $\Sigma_p$ sets and number $\mu$,  the following result was already obtained in \cite[Theorem 10]{DM}. We shall, however, give a new proof of it that is based on finding all the solutions of the equation (\ref{eq1}).
\begin{proposition} \label{thm1} For the domain $G_1$ bounded by the curve $L_1$ of (\ref{eq118}),
\begin{equation}\label{mu}
\mu=\frac{R-\sqrt{R^2-4}}{2}
\end{equation}
and $\Sigma$ consists of those points $z\in G_1$ for which $|v_{z,+}|>\mu$. Furthermore, if $z\in\Sigma$ and $\omega$ is one of the solutions of largest modulus that the equation (\ref{eq1}) has in $\mu<|w|<1$, then $\omega\in\{v_{z,+},v_{z,-}\}$. As a consequence,
\[
G_1=\Sigma_1\cup \Sigma_2\cup \Sigma_0,
\]
with $\Sigma_1$ being the image by $\psi$ of those points of $\mathbb{D}_1$ that lie exterior to both the circle $|w|=\mu$ and the circle $|w-1/R|=1/R$, and $\Sigma_2=(x_\mu,2]$, with $x_\mu=(1+\mu^2)^2-2$ (see Figure \ref{figura3}).
\end{proposition}

To prove Proposition \ref{thm1} we shall use the following fundamental result. Hereafter $\mu$ is given by (\ref{mu}) and  $\lambda(z)$ is defined by (\ref{fractional}). We shall implicitly use that $\lambda(z)$ is its own inverse and that
\begin{equation}\label{mu1}
\mu=(R-\mu)^{-1}\,.
\end{equation}

\begin{proposition}\label{proposition} Let $\varphi$ be a conformal map of $G_1$ onto $\mathbb{D}_1$.
\begin{enumerate}
\item[(a)] The function $\varphi(\psi(w))$, originally defined on $\rho<|w|<1/\rho$, admits a meromorphic continuation, denoted by $h_\varphi(w)$, to all of $\mathbb{C}\setminus \{\mu,1/\mu\}$. Moreover, $\mu$ and $1/\mu$ are both non-isolated singularities of $h_\varphi$ of ``essential type", in the sense that in every punctured neighborhood of either one of these two points, the function $h_\varphi$ attains every value of the extended complex plane.
\item[(b)] $h_\varphi\circ\lambda$ is meromorphic in $\mathbb{C}\setminus \{0\}$, and for all $k\in\mathbb{Z}\setminus\{0\}$,
 \begin{equation}\label{eq3232}
 (h_\varphi\circ\lambda)(t)=\left\{\begin{array}{ll}
 {\displaystyle\frac{1}{\overline{(h_\varphi\circ\lambda)(\overline{t}/\mu^{2k})}}}, & \mu^{2k+2}\leq |t|\leq \mu^{2k},\ k\ \mathrm{odd},\\ \\
(h_\varphi\circ\lambda)(t/\mu^{2k}), &  \mu^{2k+2}\leq |t|\leq \mu^{2k},\ k\ \mathrm{even}.
 \end{array}
 \right.
 \end{equation}
\item[(c)] For every $t_0$ with $\mu^2<|t_0|<1$, the solutions that the equation 
\[
(h_\varphi\circ\lambda)(t)=(h_\varphi\circ\lambda)(t_0)
\]
has in $0<|t|< 1$ are the elements of the two sequences $\{\mu^{4k}t_0\}_{k=0}^\infty$ and $\{\mu^{4k+2}/t_0\}_{k=0}^\infty$. Moreover,
 \begin{align}\label{derivatives}
 (h_\varphi\circ\lambda)'(\mu^{4k}t_0)&{}=-\frac{\varphi'(\psi(\lambda(t_0)))(1-\mu^4)^2\left(t_0-\mu^2/t_0\right)}{\mu^{4k+1} t_0(1-\mu t_0)^2(1-\mu^3/t_0)^2},\quad k\geq 0,
\end{align}
and if $\mu<|\lambda(t_0)|<1$, then 
\begin{equation}\label{modulo}
|\lambda(t_0)|>|\lambda(\mu^{4k}t_0)|,\quad k\geq 1.
\end{equation}
\end{enumerate}
\end{proposition}
\begin{proof}
Observe first that $v_{z,+}$ and $\overline{v}_{z,-}$ are reflections of each other about the circle $|w-1/R|=1/R$, given that $Rv_{z,+}-1$ and $R\overline{v}_{z,-}-1$ are reflections of each other about the unit circle. Since the reflection of the unit circle about $|w-1/R|=1/R$ is the circle $|w-R/(R^2-1)|=1/(R^2-1)$, we then have that $\psi$ maps
\[
\mathfrak{D}:=\{w:|w-R/(R^2-1)|>1/(R^2-1),\ |w|<1\}
\]
in a two-to-one fashion onto $G_1$, and $\partial \mathfrak{D}$ onto $L_1$. Because $\varphi$ is analytic in a neighborhood of $\overline{G}_1$, the composition $h_\varphi(w)=\varphi(\psi(w))$ makes sense and is analytic in $\overline{\mathfrak{D}}$.

On the other hand, using (\ref{mu1}), it is easy to see that $\lambda$ maps the annulus $\mu^2\leq |t|\leq 1$ conformally onto $\overline{\mathfrak{D}}$. In effect, $\lambda$ is an automorphism of the unit circle, it maps $|t|=\mu$ onto $|w-1/R|=1/R$, and preserves reflections about circles, so that it maps $|t|=\mu^2$ onto $|w-R/(R^2-1)|=1/(R^2-1)$. Thus, $h_\varphi\circ\lambda$ is analytic on $\mu^2\leq |t|\leq 1$, mapping the boundary of this annulus onto the unit circle, and since
\[
(\psi\circ\lambda)(\mu^2t)=\overline{(\psi\circ\lambda)(t)}=(\psi\circ\lambda)(\overline{t}),\quad |t|=1,
\]
we then have
\[
(h_\varphi\circ\lambda)(\mu^2t)=\frac{1}{\overline{(h_\varphi\circ\lambda)(\overline{t})}},\quad |t|=1.
\]

Hence, we can extend $h_\varphi\circ\lambda$ meromorphically to all of $\mathbb{C}\setminus\{0\}$ as specified by (\ref{eq3232}). Now, we see from (\ref{eq3232}) that $h_\varphi\circ \lambda$ maps $\mu^2\leq |t|\leq 1$ onto $\overline{\mathbb{D}}_1$ and  $\mu^4\leq |t|\leq \mu^2$ onto $\overline{\mathbb{C}}\setminus \mathbb{D}_1$, and that
\begin{equation}\label{eq1003}
(h_\varphi\circ \lambda)(\mu^4t)=(h_\varphi\circ \lambda)(t).
\end{equation}
Hence, $h_\varphi\circ \lambda$ maps every annulus $\mu^{4(k+1)}\leq |t|\leq \mu^{4k}$, $k\in \mathbb{Z}$, onto  $\overline{\mathbb{C}}$, and thus composing back with $\lambda$ (which is its own inverse) we obtain Part (a) of the proposition.

The first statement of Part (c) equally follows from (\ref{eq3232}), while from (\ref{eq1003}) we get
\begin{align*}
(h_\varphi\circ\lambda)'(\mu^{4k}t_0)&=\mu^{-4k}(h_\varphi\circ\lambda)'(t_0)
=\mu^{-4k}\varphi'(\psi(\lambda(t_0)))(\psi\circ\lambda)'(t_0),\quad j=1,2,
\end{align*}
which combined with  
\begin{align*}
(\psi\circ\lambda)'(t)&=\left[\mu^{-1}\frac{t-\mu^3}{\mu t-1}+\mu\frac{\mu t-1}{t-\mu^3}\right]'=-\frac{(1-\mu^4)^2(t-\mu^2/t)}{\mu t(1-\mu t)^2(1-\mu^3/t)^2}
\end{align*}
yields (\ref{derivatives}).

Finally, (\ref{modulo}) follows from the fact that $\lambda(t)$ maps $[-\infty,0]$ onto $[\mu,1/\mu]$ and $[0,\infty]$ onto  $[-\infty,\mu]\cup [1/\mu,\infty]$, respectively, while any other line passing through the origin gets mapped onto a circle passing through the points $\mu$ and $1/\mu$, and the two points of this circle that are closest to and farthest from the origin lie, respectively, inside $\mathbb{T}_\mu$ and outside $\mathbb{T}_1$.
\end{proof}

We can now give the
\begin{proof}[Proof of Proposition \ref{thm1}]
The fact that $\mu$ as given by (\ref{mu}) is the smallest number such that $\varphi(\psi(w))$ admits a meromorphic continuation to the annulus $\mu<|w|<1$ now emerges clearly from Proposition \ref{proposition}(a).

For a  given $z\in G_1$, we now look for the solutions that the equation (\ref{eq1}) has on $\{w:|w|<1,\ w\not= \mu\}$. Since $\lambda(t)$ is an automorphism of the unit disk, this is equivalent to finding the solutions that the equation 
\begin{equation}\label{composition}
(h_\varphi\circ\lambda)(t)=\varphi(z)
\end{equation} has on $0<|t|<1$. We have already observed in the proof of Proposition \ref{proposition} that $\psi(w)$ maps the region $\mathfrak{D}$ onto $G_1$, and that for each $z\in G_1$, the solutions to $z=\psi(\lambda(t))$ are the numbers $v_{z,\pm}$. Since $\varphi$ is univalent on $G_1$ and $\lambda(t)$ maps $\mu^2<|t|<1$ onto $\mathfrak{D}$, we see that (\ref{composition}) has exactly two solutions on $\mu^2<|t|<1$, given by $\lambda(v_{z,\pm})$ (recall $\lambda$ is its own inverse). Since $\lambda(v_{z,+})=\mu^2/\lambda(v_{z,-})$, we infer from Proposition \ref{proposition}(c) that the solutions that the equation $h_\varphi(w)=\varphi(z)$ has in $\mathbb{D}_1\setminus\{\mu\}$ are the elements of the two sequences $\{\lambda(\mu^{4k}\lambda(v_{z,\pm}))\}_{k=0}^\infty$. 

Now, $\lambda(t)$ maps the disk $|t-1/R|\leq 1/R$ onto $|w|\leq \mu$, and since $|v_{z,+}|\geq |v_{z,-}|$, we find that $h_\varphi(w)=\varphi(z)$ has solutions in $\mu<|w|<1$ if and only if $|v_{z,+}|>\mu$, in which case, by (\ref{modulo}), the solutions of largest modulus are at most two and contained in $\{v_{z,+},v_{z,-}\}$. 
\end{proof}

\section{Auxiliary lemmas}

The following lemmas have been set apart because they are rather technical and may obscure the central idea of the proof of Theorems \ref{mainthm} and \ref{thm4}.  For a first reading, we recommend the reader to trust the validity of Lemma \ref{lem2} below and move on to the next section.

\begin{lemma} \label{lem7} For every compact set $E\subset \{t:|1-Rt|<1\}$, there exist positive constants $m$ and $M$ such that for every integer $n\geq 1$,
\begin{align*}
\left |e^{-(\mu^{-1}-\mu) ts}-\mu^{-n}\lambda^n(ts/n)\right|\leq {}& \frac{Ms^2e^{-ms}}{n}, \quad t\in E,\quad 0\leq s\leq n\,.
\end{align*}
\end{lemma}
\begin{proof}
For every $z\in \mathbb{C}$, the function $\kappa(s):=|1-sz|$ is convex in $\mathbb{R}$. Hence, for $z\in U_1:=\{z:|1-z|<1\}$ and an integer $n\geq 1$, we have
\begin{equation}\label{eq102}
\left|1-\frac{sz}{n}\right|\leq \left|1-(1-|1-z|)\frac{s}{n}\right|\leq e^{-(1-|1-z|)s/n}, \quad 0\leq s\leq n.
\end{equation}

Next, suppose $t$ is such that %$|1-Rt|<1$, or equivalently, that
$|\mu^{-1}\lambda(t)|<1$ and consider the M\"{o}bius transformation
\[
\sigma_t(s):=\frac{(1-\mu^2)t}{\mu (1-\mu ts)},\quad s\in \mathbb{R}.
\]
Since
\begin{equation}\label{eq103}
\mu^{-1}\lambda(t)=1-\frac{(1-\mu^2)t}{\mu (1-\mu t)},
\end{equation}
we readily see that $\sigma_t(1)\in U_1$. Also, since $|\mu^{-1}\lambda(t)|<1$ if and only if $|1-Rt|<1$, and $\mu^{-1}(1-\mu^2) =(1-\mu^2)(1+\mu^2)^{-1}R<R$, it follows that $\sigma_t(0)=\mu^{-1}(1-\mu^2)t\in U_1$.
Then, given that $\sigma_t(\infty)=0\in\partial U_1$ and that $\sigma$ maps the real line conformally onto a circle, we conclude that $\sigma_t$ must map the segment $[0,1]$ onto a circular arc that lies inside $U_1$. Therefore, by (\ref{eq102}) and (\ref{eq103}), we have for all $ 0\leq s\leq n$ and $|1-Rt|<1$ that
\begin{align}\label{eq1004}
\begin{split}
|\mu^{-1}\lambda(ts/n)|={} &\left|1-\frac{s}{n}\,\sigma_t(s/n)\right|\leq e^{-(1-|1-\sigma_t(s/n)|)s/n}.
\end{split}
\end{align}

Now, to abbreviate, set $\alpha=\mu^{-1}-\mu$, so that for $0\leq s\leq n$ and $|1-Rt|<1$ we have
\begin{align}\label{eq105}
\begin{split}
\left|e^{-\alpha ts/n}-\mu^{-1}\lambda(ts/n)\right|={}&\left|e^{-\alpha ts/n}-\left(1-\frac{\alpha ts/n}{1-\mu ts/n}\right)\right|\\
={}& \frac{s^2}{n^2}\left|\frac{\mu\alpha t^2}{1-\mu ts/n} +\sum_{j=2}^\infty\frac{(-\alpha t)^j}{j!}\left(\frac{s}{n}\right)^{j-2}\right|\\
\leq {}& \frac{s^2}{n^2}\left(\frac{\mu\alpha |t|^2}{1-\mu |t|} +e^{\alpha |t|}-\alpha|t|-1\right).
\end{split}
\end{align}
Thus, for a compact set $E\subset \{t:|1-Rt|<1\}$, and with the constants
\[
M:=\max_{t\in E}\left(\frac{\mu\alpha |t|^2}{1-\mu |t|} +e^{\alpha |t|}-\alpha|t|-1\right),\quad m_1:=\alpha\min_{t\in E}\Re(t)>0,
\]
\[
m_2:=\min_{(t,s)\in E\times[0,n]}(1-|1-\sigma_t(s/n)|)>0,\quad m:=2^{-1}\min \{m_1,m_2\},
\]
we obtain from (\ref{eq1004}) and (\ref{eq105}) that
\begin{align*}
\left |e^{-\alpha ts}-\mu^{-n}\lambda^n(ts/n)\right|\leq {}& \frac{Ms^2}{n^2}\sum_{\ell=1}^ne^{-\alpha \Re(t)s(\ell-1)/n}\left|\mu^{-1}\lambda(ts/n)\right|^{n-\ell}\\
\leq {}& \frac{Ms^2}{n^2}\sum_{\ell=1}^ne^{-\alpha \Re(t)s(\ell-1)/n}e^{-(1-|1-\sigma_t(s/n)|)s(n-\ell)/n}\\
\leq {}& \frac{Ms^2e^{-ms}}{n}, \quad t\in E,\quad 0\leq s\leq n\,.
\end{align*}\end{proof}

Recall that $\langle x\rangle $ denotes the fractional part of $x$. It is easy to verify that
\begin{equation}\label{fractional1}
\chi(\gamma t)=\frac{\mu^3t^2}{1+\mu^2}\int_0^\infty \mu^{-4\langle\log_{\mu^4}(s/\gamma)\rangle }s e^{(\mu-\mu^{-1}) ts}ds\,,\quad \Re(t)>0,\quad \gamma>0.
\end{equation}
In particular, we see that the functions $\{\chi(nt)\}_{n=1}^\infty$ are uniformly bounded on compacts of $\Re(t)>0$.
 
\begin{lemma} \label{lem8} With $G_n(t):=\mu^{-n}\lambda^n(t)\lambda'(t)$, $n\geq 1$, and $\chi(t)$ defined by (\ref{fractional}), we have
\begin{equation}\label{eq1103}
\int_0^n \mu^{-4\langle\log_{\mu^4}(s/n)\rangle }s G'_{n+1}(st/n) ds=\frac{(n+1)(1-\mu^2)(1-\mu^4)[\chi(nt)+O(1/n)]}{\mu^4 t^2}
\end{equation}
locally uniformly on $\{t:|1-Rt|<1\}$ as $n\to\infty$.
\end{lemma}
\begin{proof}  Let $E\subset \{t:|1-Rt|<1\}$ be compact. First we notice that
\begin{align}\label{eq106}
G'_{n+1}(st/n)={} &(n+1)\mu^{-n-1}\lambda^n(ts/n)\mathfrak{L}(st/n)
\end{align}
with
\begin{align}\label{eq107}
\begin{split}
\mathfrak{L}(st/n)={}&\frac{(1-\mu^2)^2}{(1-\mu ts/n)^4}+\frac{2\mu(1-\mu^2)(ts/n-\mu)}{(n+1)(1-\mu ts/n)^4}\\
={} &(1-\mu^2)^2\left[1+O(s/n)+O(1/n)\right]
\end{split}
\end{align}
uniformly for $(t,s)\in E\times [0,n]$ as $n\to\infty$. Hence, with $\alpha=\mu^{-1}-\mu$ and given that $1\leq \mu^{-4\langle\log_{\mu^4}(s/n)\rangle }\leq \mu^{-4}$, we get
\begin{align}\label{eq108}
\begin{split}
\int_0^n \mu^{-4\langle\log_{\mu^4}(s/n)\rangle }s e^{-\alpha ts}\mathfrak{L}(st/n)ds={}&(1-\mu^2)^2\int_0^n \mu^{-4\langle\log_{\mu^4}(s/n)\rangle }s e^{-\alpha ts}ds\\
&+O(1/n)
\end{split}
\end{align}
uniformly for $(t,s)\in E\times [0,n]$ as $n\to\infty$.

Combining (\ref{eq106}), (\ref{eq107}), (\ref{eq108}) and Lemma \ref{lem7}, we readily see that there exist positive constants $M'$ and $m$ such that
\begin{align*}
&\left|\int_0^n \mu^{-4\langle\log_{\mu^4}(s/n)\rangle }s G'_{n+1}(st/n) ds-\frac{(n+1)(1-\mu^2)^2}{\mu}\int_0^n \mu^{-4\langle\log_{\mu^4}(s/n)\rangle }s e^{-\alpha ts}ds\right|\\
&\leq \frac{n+1}{\mu}\int_0^n \mu^{-4\langle\log_{\mu^4}(s/n)\rangle}s \left|\mu^{-n}\lambda^n(ts/n)-e^{-\alpha ts}\right||\mathfrak{L}(st/n)|ds +O(1)\\
&\leq M'\int_0^\infty s^3 e^{-m s}ds +O(1)
\end{align*}
uniformly in $t\in E$ as $n\to\infty$, which together with (\ref{fractional1}) (take $\gamma=n$) yields (\ref{eq1103}).

\end{proof}

\begin{lemma}\label{lem2} For  $G_n(t)$ defined as in Lemma \ref{lem8}, we have
\begin{align*}
\sum_{k=0}^\infty\mu^{4k}G_{n+1}(\mu^{4k}t)={}& -\frac{(n+1)(1-\mu^2)\left[\chi(nt)+O(1/n)\right]}{n^2t}\end{align*}
locally uniformly on $\{t:|1-Rt|<1\}$ as $n\to\infty$.
\end{lemma}
\begin{proof}
Using summation by parts we find
\begin{align*}
\sum_{k=0}^K\mu^{4k}G_{n+1}(\mu^{4k}t)={}&\frac{G_{n+1}(t)}{1-\mu^{4}}- \frac{\mu^{4(K+1)}G_{n+1}(\mu^{4K}t)}{1-\mu^{4}}\\
& +\sum_{k=0}^{K-1}\frac{\mu^{4(k+1)}}{1-\mu^{4}}\left[G_{n+1}(\mu^{4(k+1)}t)-G_{n+1}(\mu^{4k}t)\right]\,.
\end{align*}
Letting $K\to\infty $ and using Lemma \ref{lem8} we obtain
\begin{align*}
\sum_{k=0}^\infty\mu^{4k}G_{n+1}(\mu^{4k}t)={}&\frac{G_{n+1}(t)}{1-\mu^{4}}+\sum_{k=0}^{\infty}\frac{\mu^{4(k+1)}}{1-\mu^{4}}\int_{\mu^{4k}}^{\mu^{4(k+1)}}
\frac{\partial G_{n+1}(st)}{\partial s} ds\\
={}& \frac{G_{n+1}(t)}{1-\mu^{4}} -\frac{\mu^4t}{(1-\mu^{4})n^2}\int_0^n \mu^{-4\langle\log_{\mu^4}(s/n)\rangle }s G'_{n+1}(st/n) ds\\
={}& -\frac{(n+1)(1-\mu^2)\left[\chi(nt)+O(1/n)\right]}{n^2t}
\end{align*}
locally uniformly on $\{t:|1-Rt|<1\}$ as $n\to\infty$.
\end{proof}

Before passing to the next section, we observe that 
\begin{equation}\label{chi}
\chi(nt)- \chi((n+1)t)=O(1/n)
\end{equation}
locally unformly on $\Re(t)>0$. Indeed, from the representation (\ref{fractional1}) we find
\begin{align*}
&\chi(nt)- \chi((n+1)t)\\&{}=\frac{\mu^3t^2}{1+\mu^2}\int_0^\infty \mu^{-4\langle\log_{\mu^4}(s/n)\rangle }s e^{-\alpha ts}\left(1-\frac{(n+1)^2}{n^2}e^{-\alpha ts/n}\right)ds,
\end{align*}
and the claim follows since 
\[
 |1-e^{-\alpha ts/n}|=\frac{\alpha s}{n}\left|\int_0^te^{-\alpha
 sz/n}dz\right|\leq \frac{\alpha s|t|}{n},\quad  \Re(t)>0,
\]
where the integration is taken along the segment from $0$ to $t$.
\section{Proofs}\label{Proofs}
\begin{proof}[Proof of Theorems \ref{mainthm} and \ref{thm4}]
We have already observed in the proof of Proposition \ref{proposition} that the function $\psi(\lambda(t))$ maps the annulus $\mu^2<|t|<1$ onto $G_1$, and that for every $z\in G_1$, the only solutions that the equation $z=\psi(\lambda(t))$ has in said annulus are the numbers $t_{z,\pm}=\lambda(v_{z,\pm})$, which satisfy $t_{z,+}=\mu^2/t_{z,-}$. If we now make the change of variables $w=\lambda(\zeta)$ in the integral representation given by Proposition  \ref{cor1}, use Proposition \ref{proposition}(c) and  the residue theorem, we get that for every integer $N\geq 1$, $z\in G_1\setminus\{-2,2\}$, and $n$ sufficiently large, 
\begin{align*}
\begin{split}
&(n+1)^{-1/2}p_n(z)\\
={}&  \frac{\varphi'(z)}{2\pi i}\oint_{\mathbb{T}_1} \frac{[\lambda(\zeta)]^{n}\lambda'(\zeta) [1+K_n(\lambda(\zeta))] d\zeta}{(h_\varphi\circ\lambda)(\zeta)-(h_\varphi\circ\lambda)(t_{z_\pm})}\\
={}&\frac{(1-\mu t_{z,\pm})^2(1-\mu t_{z,\mp})^2\mu^{n+1}}{(1-\mu^4)^2\left(t_{z,\pm}-t_{z,\mp}\right)}\left[\zeta\sum_{k=0}^{N-1}\mu^{4k}G_n(\mu^{4k}\zeta)[1+K_n(\lambda(\mu^{4k}\zeta))]\right]_{\zeta=t_{z,\pm}}^{\zeta=t_{z,\mp}}\\
&+\frac{\varphi'(z)}{2\pi i}\int_{|\zeta|=\mu^{4N-2}}\frac{[\lambda(\zeta)]^n\lambda'(\zeta)[1+K_n(\lambda(\zeta))]d\zeta}{(h_\varphi\circ\lambda)(\zeta)-(h_\varphi\circ\lambda)(t_{z,\pm})},
\end{split}
\end{align*}
where $G_n(\zeta)=\mu^{-n}\lambda^n(\zeta)\lambda'(\zeta)$, and $K_n(\lambda(t))=O(r^{2n})$ uniformly on $|t|\leq 1$ as $n\to\infty$ for every $r\in(\rho,1)$.  Here we are using the notation $\left[F(\zeta)\right]_{\zeta=a}^{\zeta=b}=F(b)-F(a)$.

Since $|\lambda(\zeta)|= 1$ for $|\zeta|= 1$ and $|(h_\varphi\circ\lambda)(\zeta)|=1$ for $|\zeta|=\mu^{4N-2}$, we have 
\[
\left|\frac{1}{2\pi i}\int_{|\zeta|=\mu^{4N-2}}\frac{[\lambda(\zeta)]^n\lambda'(\zeta)[1+K_n(\lambda(\zeta))]d\zeta}{(h_\varphi\circ\lambda)(\zeta)-(h_\varphi\circ\lambda)(t_{z,\pm})}\right|\leq \frac{O(\mu^{4N})}{1-|\varphi(z)|}\underset{N\to\infty}{\to} 0,
\]
and so we arrive at the following representation, valid for all $z\in G_1\setminus\{-2,2\}$ and $n$ large: 
\begin{align}\label{1001}
\begin{split}
&(n+1)^{-1/2}p_n(z)\\
={}&\frac{(1-\mu t_{z,\pm})^2(1-\mu t_{z,\mp})^2\mu^{n+1}}{(1-\mu^4)^2\left(t_{z,\pm}-t_{z,\mp}\right)}\left[\zeta\sum_{k=0}^{\infty}\mu^{4k}G_n(\mu^{4k}\zeta)[1+O(r^{2n})]\right]_{\zeta=t_{z,\pm}}^{\zeta=t_{z,\mp}},
\end{split}
\end{align}
where the constant involved in the $O(r^{2n})$ term above is independent of $n$ and $z$. 

In the above calculations, the restriction that $z\not\in\{-2,2\}$ is a technical one to avoid dealing with double poles in the residue computations. But of course, by the analyticity of the functions involved, the same estimates remain true when letting $z\to\pm 2$. 

Notice that since $\lambda(t)$ maps the disk $\{t:|t-1/R|<1/R\}$ onto $|w|<\mu$, we have
 that for $k\geq 0$ and $t$ lying in that disk, 
\begin{equation}\label{eqq}
|G_n(\mu^{4k}t)|=\mu^{-n}|\lambda(\mu^{4k}t)|^n|\lambda'(\mu^{4k}t)|\leq \max_{|t-1/R|\leq1/R}|\lambda'(t)|.
\end{equation}

Now, it is easy to verify that the function $\psi(\lambda(t))$ maps the region $D$ defined by (\ref{domainD})
onto the interior of $\Sigma_0$, that $D\subset \{t:\mu^2<|t|<1\}$, and that $D$ is symmetric with respect to the circle $|t|=\mu$.  
Hence, if we make the replacement $t=t_{z,\pm}$ in (\ref{1001}), then $t_{z,\mp}=\mu^2/t$ and $z=\psi(\lambda(t))$,  and Theorem \ref{mainthm} follows immediately by combining (\ref{1001}), (\ref{eqq}), (\ref{chi}), and Lemma \ref{lem2}. 

We now prove the equality (\ref{partb}) in Theorem \ref{thm4}, which we do locally by showing that every $z_0\in [x_\mu,2]$ has a neighborhood $B_\epsilon(z_0):=\{z:|z-z_0|<\epsilon\}$ such that (\ref{partb}) holds uniformly on $z\in B_\epsilon(z_0)\cap  [x_\mu,2]$. The proof of (\ref{parta}) is omitted as it follows in a very similar way.

First, observe that as $z$ varies over the interval $[x_\mu,2]$, the two solutions $v_{z,\pm}$ of the equation $z=\psi(w)$ satisfy $v_{z,+}=\overline{v}_{z,-}$ and they vary along that arc of the circle $|w-1/R|=1/R$ that lies on $\mu\leq |w|<1$. Parametrizing the upper half of this arc by setting 
\begin{equation}\label{eqtttt}
 v_{z,+}=R^{-1}+R^{-1}e^{i\theta}=2R^{-1}e^{i\theta/2}\cos(\theta/2), \quad 0\leq \theta\leq \arccos(x_\mu/2),
\end{equation}
we see that $|v_{z,+}|=2R^{-1}|\cos(\theta/2)|$ varies from $\mu$ to $2/R$, and equals $\mu$ exactly when $z=x_\mu$.  

For a point $z_0\in (x_\mu,2]=\Sigma_2$, we have that $|v_{z_0,\pm}|=|\lambda(t_{z_0,\pm})|>\mu$. In view of the inequality (\ref{modulo}) of Proposition \ref{proposition}(c), this allows us to find numbers $\varrho_1$ and $\varrho_2$ such that
\[
|\lambda(\mu^{4k}t_{z_0,\pm})| < \varrho_1<\varrho_2<|v_{z_0,\pm}|, \quad k\geq 1.
\]
Hence, if $\epsilon>0$ is chosen sufficiently small, we can guarantee that
\begin{equation}\label{eqttt}
|\lambda(\mu^{4k}t_{z,\pm})| < \varrho_1<\varrho_2<|v_{z,\pm}|, \quad \ z\in B_\epsilon(z_0), \ k\geq 1,
\end{equation}
so that the dominant term in the right-hand side of (\ref{1001}) corresponds to $k=0$ and we get from that for all $z\in B_\epsilon(z_0)$ 
\begin{align}\label{1001001}
\begin{split}
\frac{p_n(z)}{\sqrt{n+1}}={}&\frac{(1-\mu t_{z,\pm})^2(1-\mu t_{z,\mp})^2\mu\left( t_{z,\mp} \lambda'(t_{z,\mp})v^n_{z,\mp}-t_{z,\pm} \lambda'(t_{z,\pm})v^n_{z,\pm}\right)}{(1-\mu^4)^2\left(t_{z,\pm}-t_{z,\mp}\right)}\\
&+O(\varrho_2^nr^{2n})+O(\varrho_1^n)\\
={}&\frac{v^n_{z,+}}{\psi'(v_{z,+})}+\frac{v^{n}_{z,-}}{\psi'(v_{z,-})}+O(\varrho_2^nr^{2n})+O(\varrho_1^n)
\end{split}
\end{align}
uniformly for $z\in B_\epsilon(z_0)$ as $n\to\infty$. To get this last equality we have used (\ref{derivatives}), and in case $z=z_0=2$, the estimate is to be understood in a limiting sense.

If $z_0=x_\mu$, we have $|v_{z_0,\pm}|=\mu$ and the points $t_{z_0,\pm}$ are located precisely where the circles $|t|=\mu$ and $|t-1/R|=1/R$ intersect at. Hence, if we pick $\epsilon>0$ sufficiently small, we can guarantee that as $z$ varies over $B_{\epsilon}(x_\mu)$, the points $\mu^4t_{z,\pm}$ vary over some fixed compact subset of $|t-1/R|<1/R$, so that (\ref{eqq}) holds true for $t=t_{z,\pm}$  and $k\geq 1$, and we get once again from (\ref{1001}) and Lemma \ref{lem2} that 
\begin{align}\label{eqtt}
\begin{split}
(n+1)^{-1/2}p_n(z)={}&\frac{v^n_{z,+}}{\psi'(v_{z,+})}+\frac{v^{n}_{z,-}}{\psi'(v_{z,-})}+O(\mu^n/n)
\end{split}
\end{align}
uniformly for $z\in B_{\epsilon}(x_\mu)$ as  $n\to\infty$. Using (\ref{eqtttt}) and (\ref{eqttt}) one can easily verify that for $z\in B_{\epsilon}(x_\mu)\cap[x_\mu,2]$, (\ref{1001001}) and (\ref{eqtt}) transform into (\ref{partb}).
\end{proof}
\begin{proof}[Proof of Theorem \ref{thm8}] 
In view of Theorem \ref{mainthm}, it suffices to show that for a subsequence $\{n_k\}\subset \mathbb{N}$,  the functions $\chi(n_k t)$ converge normally on $D$ as $k\to\infty$ if and only if (\ref{limite}) holds true for some $q\in[0,1)$. 

The ``if" part, that is, $\chi(n_k t)\to \chi(\mu^{4q} t)$ whenever (\ref{limite}) holds, follows directly from the representation (\ref{fractional1}). For the ``only if" part,  suppose that the functions $\chi(n_kt)$ converge as $k\to\infty$, but that (\ref{limite}) holds true for no $q\in [0,1)$. Then, we can find two subsequences  of $\{n_k\}$, say $\mathcal{N}_1$ and  $\mathcal{N}_2$, such that 
\[
\lim_{\underset{n\in\mathcal{N}_1}{n\to\infty}}e^{2\pi i \log_{\mu^4}(n)}=e^{2\pi i q},\quad \lim_{\underset{n\in\mathcal{N}_2}{n\to\infty}}e^{2\pi i \log_{\mu^4}(n)}=e^{2\pi i p}
\]
for some $0\leq q<p<1$.  But this and the convergence of $\chi(n_kt)$ leads to a contradiction, for we shall now prove that 
\[
\chi(\mu^{4q}t)-\chi(\mu^{4q+2}/t)\not\equiv \chi(\mu^{4p}t)-\chi(\mu^{4p+2}/t)
\]
(or equivalently, that $f_q\not=f_p$) for $0\leq q<p<1$. 
 
Making $t=\mu^{4x+1}$, $\varrho=\mu^{4}$, $\mathfrak{q}=q+1/4$, and defining  
\[
g_\mathfrak{q}(x):=\chi(\varrho^{\mathfrak{q}+x})-\chi(\varrho^{\mathfrak{q}-x}), \quad x\in \mathbb{R},
\]
we see that it suffices to show that $g_\mathfrak{q}-g_\mathfrak{p}\not= 0$ for $1/4\leq \mathfrak{q}<\mathfrak{p}<5/4$. The function
\[
g_\mathfrak{q}(x)=\sum_{k=-\infty}^{\infty}\varrho^{k+\mathfrak{q}+x}e^{-\alpha\varrho^{k+\mathfrak{q}+x}}-
\sum_{k=-\infty}^{\infty}\varrho^{k+\mathfrak{q}-x}e^{-\alpha\varrho^{k+\mathfrak{q}-x}}, \quad \alpha=\mu^{-1}-\mu>0,
\] is analytic and $1$-periodic on $\mathbb{R}$. Using the Poisson summation formula, we find its Fourier expansion to be
\[
g_\mathfrak{q}(x)-g_\mathfrak{p}(x)=8\sum_{n=1}^{\infty}\Re(c_n)\sin(\pi n(\mathfrak{p}-\mathfrak{q}))\sin(2\pi nx),\quad x\in\mathbb{R},
\]
where
\[
c_n=\int_{-\infty}^{\infty}\varrho^{t}
e^{-\alpha\varrho^{t}}e^{i\pi n(2t-\mathfrak{q}-\mathfrak{p})}dt,
\]
and we only need to show that there is at least one integer $n\geq 1$ with $\Re(c_n)\not=0$.

Making the change of variable $v=\alpha\varrho^t$ and using Stirling's approximation formula for the Gamma function we find
\begin{align*}
c_n={}&\frac{e^{-i\pi n(\mathfrak{q}+\mathfrak{p})}\Gamma\left(1+2\pi ni/\ln\varrho\right)}{\alpha^{1+2\pi ni/\ln\varrho}\ln\varrho}\\
={}&\frac{\left|\Gamma\left(1+2\pi ni/\ln\varrho\right)\right|(1+O(1/n))}{\alpha\ln\varrho}\cdot e\left(\frac{n\ln n}{\ln\varrho}-\sigma n+\frac{1}{8}\right),\quad n\geq 1,
\end{align*}
where $e(x):=\exp(2\pi i x)$ and
\[
\sigma:=\frac{\mathfrak{q}+\mathfrak{p}}{2}+\frac{1+\ln\alpha-\ln(2\pi/\ln\varrho)}{\ln\varrho}\,.
\]
It follows that $c_n\not=0$ for all $n\geq 1$. Moreover, taking the quotient $c_{n+1}/c_n$, we find that if all the $c_n$'s are purely imaginary, then
\[
\Im \left\{e\left(\frac{\ln(1+\frac{1}{n})^n}{\ln\varrho}+\frac{\ln(n+1)}{\ln\varrho}-\sigma\right)(1+O(1/n))\right\}=0,\quad n\geq 1,
\]
contradicting the fact that the sequence $\{\ln n/\ln\varrho\}_{n=1}^\infty$ is dense in $[0,1]$ modulo 1.

This last statement follows from the more general and easily verifiable observation that $\{\langle a_n\rangle\}_{n=1}^\infty$ is dense in $[0,1]$ whenever $\lim_{n\to\infty}a_n=+\infty$ and $\lim_{n\to\infty}a_{n+1}-a_n=0$.
\end{proof}

We finish this section by briefly indicating how to obtain Proposition \ref{cor1} from the recent results of \cite{series}.

Let $\kappa_n$ denote the leading coefficient of $p_n$, and let $P_n(z)=\kappa_n^{-1}p_n(z)$ be the $n$th monic orthogonal polynomial. Carleman proved in \cite[Satz IV]{Carleman} that
\begin{equation}\label{leadingcoef}
\kappa_n=\sqrt{n+1} [\phi'(\infty)]^{n+1}(1+O(\rho^{2n}))
\end{equation}
as $n\to\infty$.

Let $r$ be a number such that  $\rho<r<1$. For this $r$ and every integer $n\geq 0$, we construct a sequence of functions $\{f_{n,k}(z)\}_{k=0}^\infty$ as specified by equations (8) and (9) of \cite{series}. In \cite[Proof of Theorem 1.1]{series}, it is shown that for $n$ large enough, the series $\sum_{k=0}^{\infty}f_{n,2k+2}(z)$ and $\sum_{k=0}^{\infty}f_{n,2k+1}(z)$ converge absolutely and locally uniformly for $z \in\Omega_{\rho}\backslash L_{1/r}$ and $z \in G_{1/\rho}\backslash L_r$, respectively, and that
\begin{equation}\label{eq110}
\sum_{k=0}^{\infty}f_{n,2k+2}(z) = O(r^{2n}), \quad \sum_{k=0}^{\infty}f_{n,2k+1}(z)= O(r^n)
\end{equation}
locally uniformly as $n\to\infty$ in their respective domains of definition.  

From the very definition of $f_{n,2k+2}$ in \cite[equation (9)]{series},  we see that for $k\geq 0$, $f_{n,2k+2}(\psi(w))$ has an analytic continuation $f_{n,2k+2}^*(w)$ to $\mathbb{C}\setminus \mathbb{T}_{1/r}$ given by
\begin{equation}\label{eqqtt1}
f_{n,2k+2}^*(w)=\frac{1}{2 \pi i} \oint_{\mathbb{T}_{1/r}} \frac{f_{n,2k+1}(\psi(t)) t^{-n-1} dt}{t-w}, \quad |w|\neq 1/r,
\end{equation}
and if we define
\begin{equation*}
F_n(w):=\sum_{k=0}^{\infty} f_{n,2k+2}^* (w),\quad |w|\not=1/r,\quad n\geq 0,
\end{equation*}
then by (\ref{eq110}) and (\ref{eqqtt1}),
\begin{equation}\label{eq2}
F_n(w)= O(r^{2n})\quad \text{and}\quad  F_n'(w)=O(r^{2n})
\end{equation}
locally uniformly on $|w|\neq 1/r$ as $ n \rightarrow \infty$.

According to Theorem 1.1 of \cite{series}, for $z\in G_r$ and $n$ large,
\begin{align}\label{eq1120}
(n+1) [\phi'(\infty)]^{n+1} P_n(z) &=  \frac{\varphi'(z)}{2\pi i} \oint_{L_r} \frac{\left(\sum_{k=0}^{\infty}f_{n,2k}(\xi) \right)\varphi'(\xi) \phi(\xi)^{n+1} d\xi}{[\varphi(\xi)-\varphi(z)]^2}\,.
\end{align}
Deforming the contour of integration $L_r$ into $L_1$  does not change the value of this last integral for values of $z \in G_r$, and leaves a function that is analytic in $G_1$. By the uniqueness of the analytic continuation and integrating by parts in (\ref{eq1120}) followed by the change of variables $\xi=\psi(w)$, we obtain
\begin{align}\label{eq1018}
\begin{split}
P_n(z) &=   \frac{\varphi'(z)}{(n+1) [\phi'(\infty)]^{n+1}2\pi i} \oint_{\mathbb{T}_1} \frac{[(1+F_n(w)) w^{n+1}]' dw}{\varphi(\psi(w))-\varphi(z)} \\
	&=  \frac{\varphi'(z)}{ [\phi'(\infty)]^{n+1}2\pi i} \oint_{\mathbb{T}_1} \frac{w^{n}(1+K^*_n(w))  dw}{\varphi(\psi(w))-\varphi(z)},\quad  z \in G_1,
\end{split}
\end{align}

with $K_n^*(w)= F_n(w) + w F_n'(w)/(n+1)$, and so multiplying (\ref{eq1018}) by $\kappa_n$, letting
\[
K_n(w):=\frac{\kappa_nK_n^*(w)}{\sqrt{n+1} [\phi'(\infty)]^{n+1}}+\frac{\kappa_n}{\sqrt{n+1} [\phi'(\infty)]^{n+1}}-1\,,
\]
and using  (\ref{eq2}) and (\ref{leadingcoef}), we arrive at Proposition \ref{cor1}.

\bibliographystyle{amsplain}

\end{document}